\documentclass{article}

\usepackage{amsmath}
\usepackage{amssymb}
\usepackage{bussproofs}
\usepackage{amsthm}
\usepackage{stmaryrd}

\usepackage{times}

\DeclareMathOperator{\dom}{dom}

\newtheorem{lemma}{Lemma}[section]
\newtheorem{theorem}{Theorem}[section]
\newtheorem{definition}{Definition}[section]

\title{Priority Arguments and Epsilon Substitutions}
\author{Henry Towsner}
\date{\today}

\begin{document}
\bibliographystyle{amsalpha}
\maketitle

\section{Introduction}
While the priority argument has been one of the main techniques of recursion theory, it has seen only a few applications to other areas of mathematics \cite{martin,solovay}.  One possibility for another such application was pointed out by Kreisel: Hilbert's $\epsilon$-substitution method, a technique for proving the $1$-consistency of theories.  Kreisel's observation was that the proof that the method works \cite{Ackermann1940} bears a striking resemblence to the structure of a traditional finite injury priority argument.

Such a connection might have benefits for both fields.  The $\epsilon$-substitution method has powerful extensions \cite{AraiEpsilonID1,AraiEpsilonMahlo,AraiEpsilonWC} which might provide new tools for solving difficult recursion theoretic problems.  In the other direction, the most popular proof theoretic technique for proving $1$-consistency results, cut-elimination, has bogged down in technical details, and new ideas are neeed to make ordinal analytic results more accessible.

Unfortunately, Kreisel's observation has been difficult to turn into a concrete argument.  After Yang \cite{Yue}, the reason is clear: the success of all finite injury priority arguments is exactly enough to prove the $1$-consistency of the weak theory $I\Sigma_1$, and therefore finite injury arguments cannot be sufficient to prove the consistency of stronger theories.  Using a general framework for priority arguments developed by Lerman and Lempp \cite{ll1,ll2,ll3}, Yang goes on to show that arguments on the $n$-th level of their hierarchy of priority arguments are equivalent to the $1$-consistency of $I\Sigma_n$, and so it requires the full $\omega$ levels of that hierarchy to give $1$-consistency for all of first-order arithmetic.

The better known infinite injury and monster injury priority arguments belong to the second and third levels of this hierarchy, and, as the name ``monster'' suggests, going to higher levels becomes impractical without some kind of general framework.  The Lerman-Lempp framework is one of several that have been proposed \cite{ash1,ash2,knight,gslaman}.  One technique, usually described using ``workers on many levels,'' originally developed by Harrington, has been extended to hyperarithmetic levels.

We show in this paper that, if one is prepared to use one of these frameworks to describe the necessary priority argument, that the $\epsilon$-substitution method can be proven to work using a priority argument.  We follow Yang in using the Lerman-Lempp framework, although we know of no reason that other frameworks would not work just as well.

Currently, those few priority arguments that have been extended to hyperarithmetic levels have a fixed ordinal height $\alpha$.  This paper and Yang's suggest that this corresponds to the $1$-consistency of Peano Arithmetic plus transfinite induction up to a particular ordinal.  The $\epsilon$-substitution has difficult but reasonably well-understood extensions to systems like $ID_1$ \cite{Mints03,AraiEpsilonID1a}, a system which adds a least fixed point to arithmetic, and (less well-understood) extensions to even stronger systems \cite{AraiEpsilonMahlo,AraiEpsilonWC}.  We hope that these results can also be translated into the priority argument environment, giving a stronger recursion theoretic technique which might be capable of answering unsolved questions.

In the hopes of making the proof more accessible, we abandon the standard terminology of the $\epsilon$-substitution method for more conventional terminology.  We work in a quantifier-free language with Skolem functions: function symbols of the form $c_{\exists x. \phi[x,\vec y]}$ where $\phi$ is quantifer-free and $\vec y$ is a sequence of variables.  A term $c_{\exists x. \phi[x,\vec y]}(\vec t)$ is intended to represent a value $n$ such that $\phi[n,\vec t]$ holds, if there is such an $n$.  Note that we allow nesting, to represent $\Sigma_n$ formulas for arbitrary $n$; for instance terms like $c_{\exists x. \phi[c_{\exists y. \psi[x,y,\vec z]},\vec w]}$ are allowed.

An $\epsilon$-substitution is just a partial model for this language, providing an interpretation for the value of some Skolem functions when evaluated at some points.  Such partial models may not satisfy all axioms, but we will be interested in satisfying only finitely many axioms at a time.  When an $\epsilon$-substitution fails to satisfy some axiom, it will always be possible to repair this in a canonical way by extended the substitution.  The act of doing so, however, may force us to remove some other elements, since changing the value of one term may alter the interpretation of others.

The Lerman-Lempp framework for priority arguments uses a tower of trees, where the $n$-th tree controls conditions guiding $\Sigma_n$ properties.  We choose branches in the tree in stages, with each stage corresponding to a step in our construction.  The important idea is that the branches we choose stabilize enough to give a well-formed construction.  For instance, at the bottom level are $\Sigma_1$ properties; in our case, these form a tree where we can only change once: when we first reach a node with one of these conditions, unless our construction already witnesses the $\Sigma_1$ case, we assume a $\Pi_1$ outcome.  If at a later stage we discover a witness, we backtrack and choose a different branch, abandoning some of our progress through the tree.  But, having been witnessed, the $\Sigma_1$ outcome cannot change, so eventually we achieve a path through this tree.  The next tree controls $\Sigma_2$ properties, which can change back and forth repeatedly; the key to the proof will be that, relative to the first tree, the second tree has controlled backtracking: that is, except when we backtrack in the first tree, the second tree behaves like the first tree.  But since we can control the backtracking in the first tree, this gives an indirect control on the path we construct in the second tree.  This process is then repeated to give enough control on all the trees to prove that the construction we want is well-behaved.

Unlike a typical priority argument, our setting is finitary.  While this changes the phrasing of some arguments, the underlying concerns are the same: in a usual priority argument, we must arrange infinitely many conditions so that they have order type $\omega$, while in this case, we must arrange finitely many (where some appear multiple times) so that they eventually run out.  This proof could be modified to work with countably many conditions---for instance, all possible conditions---and therefore to prove that there is a recursively enumerable $\epsilon$-substitution assigning correct values to all rank $1$ Skolem functions (that is, all Skolem functions for $\Sigma_1$ formulas).

Rather than literally following the $H$-process, we prove termination of a modified process derived from our construction.  The primary difference is that in certain situations we add additional information to our $\epsilon$-substitution whose correctness is witnessed even if there is no axiom compelling us to do so.  This turns out to better match our construction since it means we can decide locally, by examining only the $\epsilon$-substitution, whether that information is present, rather than having to know what happened at previous stages to figure out whether it might have been added at some point.

\section{Skolem Functions and $\epsilon$-Substitutions}
In this section, we present a simplified version of the $\epsilon$-substitution method.  For the standard presentation, as well as those lemmas whose proofs we have omitted, see \cite{MintsTupailoBuchholz1996}.

We work in a Skolemized version of first-order arithmetic.
\begin{definition}
  Let $\mathcal{L}_{0}$ be the ordinary language of first-order arithmetic.  In particular, it contains predicate symbols for each primitive recursively definable relation, and the function symbols $0$ and $\mathbf{S}$ (and no others).

  Given a language $\mathcal{L}$ define the Skolemization $\mathcal{L}'$ by adding, for each $\Sigma_1$ formula $\exists x. \phi[x,y_1,\ldots,y_k]$ such that $\phi[x,0,\ldots,0]$ contains no closed Skolem terms and $y_1,\ldots,y_k$ includes all free variables besides $x$ in $\phi$, add a $k$-ary \emph{Skolem function} $c_{\exists x. \phi[x,y_1,\ldots,y_k]}$.

  Let $\mathcal{L}_{n+1}:=\mathcal{L}_{n}'$, and let $\mathcal{L}_{\omega}:=\bigcup\mathcal{L}_n$.  Let $\mathcal{L}\epsilon$ be the quantifier-free part of $\mathcal{L}_\omega$.

A formula or term $e$ has \emph{rank} $n$, written $rk(e)=n$, if it belongs to $\mathcal{L}_n$ but no $\mathcal{L}_m$ for $m<n$.
\end{definition}

Note that $\exists x. \phi[x,y_1,\ldots,y_n]$ may contain Skolem functions which depend on $x$.

\begin{definition}
  Within $\mathcal{L}\epsilon$, we take $\exists x. \phi[x,\vec t]$ to be an abbreviation for $\phi[c_{\exists x. \phi}(\vec t),\vec t]$ and $\forall x. \phi[x,\vec t]$ to be an abbreviation for $\phi[c_{\exists x. \neg\phi}(\vec t),\vec t]$.
\end{definition}

\begin{definition}
  The only rule of $PA\epsilon$ is modus ponens.  The axioms are:
\begin{enumerate}
\item All propositional tautologies of the language $\mathcal{L}\epsilon$
\item All substitution instances of the defining axioms for predicate constants
\item Equality axioms $t=t$ and $s=t\rightarrow\phi[s]\rightarrow\phi[t]$
\item Peano axioms $\neg \mathbf{S}t=0$ and $\mathbf{S}s=\mathbf{S}t\rightarrow s=t$
\item Critical formulas:
  \begin{itemize}
  \item $\phi[t]\rightarrow\exists x. \phi[x]$
  \item $\phi[0]\wedge\neg\phi[t]\rightarrow\exists x. (\phi[x]\wedge\neg\phi[\mathbf{S}x])$
  \item $\neg s= 0\rightarrow \exists x. s=\mathbf{S}x$
\end{itemize}
\end{enumerate}
\end{definition}

This is a standard axiomization of Peano Arithmetic, except that $\exists x. \phi[x]$ is an abbreviation for a statement about Skolem terms.  Note that all critical formulas have a general form $\phi\rightarrow\psi[c]$ for a Skolem term $c$; we will sometimes make reference to this general form for an arbitrary critical formula.

\begin{theorem}
  If there is a proof of a closed formula $\phi$ in $PA$ then there is a proof of $\phi$ in $PA\epsilon$ (where quantifiers are interpreted as abbreviations) containing only closed formulas.
\end{theorem}

From here on, we assume that all formulas are closed (since in the Skolemized language there is no need for free variables).

\subsection{$\epsilon$-Substitutions}
We will be interested in particular partial models of formulas in $\mathcal{L}\epsilon$ assigning values to finitely many values of the Skolem functions.  We will only assign values to predicates of the form $c_{\exists x.\phi[x,\vec y]}(\vec t)$ where each $t_i$ is a natural number, and will assign either a natural number $u$ (asserting that $\phi[u,\vec t]$ holds) or a default value $?$ (leaving open the possibility that $\forall x. \neg\phi[x,\vec t]$).

\begin{definition}
  A \emph{canonical} term is a term of the form $c(\vec t)$ where $c$ is a Skolem function and each $t_i$ is a numeral.
\end{definition}
To keep some continuity with other work in the area, we call these models $\epsilon$-substitutions:
\begin{definition}
\item An \emph{$\epsilon$-substitution} is a function $S$ such that:
\begin{itemize}
\item The domain of $S$ is a set of canonical terms
\item If $e\in\operatorname{dom}(S)$ then $S(e)$ is either a numeral or the symbol $?$
\end{itemize}

An $\epsilon$-substitution is \emph{total} if its domain is the set of all canonical terms.
\end{definition}

We will frequently have a non-total $\epsilon$-substitution which we wish to take to be ``complete''
: that is, we wish to assign the default value to every canonical term not specifically assigned some other value.
\begin{definition}

The standard extension $\overline{S}$ of an $\epsilon$-substitution $S$ is given by
\[\overline{S}:=S\cup\{(e,?)\mid e\not\in \dom(S)\}\]
\end{definition}

\begin{definition}
  If $t$ is a term, we extend the function $S$ to define $\hat S(t)$ on arbitrary terms by induction on $t$, and also to sequences of terms:
  \begin{itemize}
  \item If $\vec s$ is the sequence $s_1,\ldots,s_k$, set $\hat S(\vec s):=\hat S(s_1),\ldots,\hat S(s_k)$
  \item If $t$ is a non-canonical Skolem term of the form $c(\vec s)$ and for some $i$, $\hat S(s_i)$ is not a numeral then $\hat S(t):=c(\hat S(\vec s))$
  \item If $t$ is a non-canonical Skolem term of the form $c(\vec s)$ and for every $i$, $\hat S(s_i)$ is a numeral then $\hat S(t):=\hat S(c(\hat S(\vec s)))$
  \item If $t$ is a canonical Skolem term in the domain of $S$ and $S(t)=?$ then $\hat S(t):=0$
  \item If $t$ is a canonical Skolem term in the domain of $S$ and $S(t)\in\mathbb{N}$ then $\hat S(t):=S(t)$
  \item If $t$ is a canonical Skolem term not in the domain of $S$ then $\hat S(t):=t$
  \item $\hat S(\mathbf{S}t):=\mathbf{S}\hat S(t)$
  \item $\hat S(0):= 0$
  \end{itemize}
\end{definition}

\begin{definition}
  \begin{itemize}
  \item If $\phi$ is an atomic formula $Rt_1\cdots t_n$ then $S\vDash Rt_1\cdots t_n$ iff $\hat S(t_i)$ is a numeral for each $i$ and $R\hat S(t_1)\cdots \hat S(t_n)$ holds in the standard model
  \item If $\phi$ is a negated atomic formula $\neg Rt_1\cdots t_n$ then $S\vDash \neg Rt_1\cdots t_n$ iff $\hat S(t_i)$ is a numeral for each $i$ and $\neg R\hat S(t_1)\cdots \hat S(t_n)$ holds in the standard model
  \item $S\vDash\phi\wedge\psi$ iff $S\vDash\phi$ and $S\vDash\psi$
  \item $S\vDash\neg(\phi\wedge\psi)$ iff $S\vDash\neg\phi$ or $S\vDash\neg\psi$
  \item $S\vDash\phi\vee\psi$ iff $S\vDash\phi$ or $S\vDash\psi$
  \item $S\vDash\neg(\phi\vee\psi)$ iff $S\vDash\neg\phi$ and $S\vDash\neg\psi$
  \end{itemize}
\end{definition}

The unusual handling of negation is necessary because if $S$ is not total, some formulas may be indeterminate.

\begin{definition}
  $S$ \emph{decides} $\phi$ if $S\vDash\phi$ or $S\vDash\neg\phi$.

  $unev(\phi,S)$, the set of terms in $\phi$ not evaluated by $S$, consists of terms of the form $c(\vec s)$ such that $\hat S(s_i)$ is a numeral for each $i$, but $\hat S(c(\vec s))$ is not a numeral.
\end{definition}

\begin{lemma}
  \begin{itemize}
  \item If $S\vDash\phi$ then $S\not\vDash\neg\phi$.
  \item If $S$ is total then $S$ decides all closed formulas.
  \item If $S$ does not decide a closed formula $\phi$ then $unev(\phi,S)$ is non-empty
  \end{itemize}
\end{lemma}

\begin{definition}
  \[S_{\leq r}:=\{(e,u)\in S\mid rk(e)\leq r\}\]
\end{definition}

\begin{lemma}
  If $S$ and $S'$ have the same domain and same values for Skolem functions of rank $\leq r$ (that is, $S_{\leq r}=S'_{\leq r}$) and $\phi$ contains only Skolem functions of rank $\leq r$ then $S\vDash\phi$ iff $S'\vDash\phi$.
\end{lemma}

The purpose of $\epsilon$-substitutions is the following theorem:
\begin{theorem}
  Suppose that for every proof of a formula $\phi$ in $PA\epsilon$, there is an $\epsilon$-substitution $S$ such that $S\vDash\phi$.  Then Peano Arithmetic is $1$-consistent.
\end{theorem}
\begin{proof}
  If $PA\vdash\exists x.\phi[x, \vec t]$ then $PA\epsilon\vdash\phi[c_{\exists x. \phi}(\vec t),\vec t]$ where the terms in $\vec t$ are numerals.  By assumption, there is an $S$ such that $S\vDash\phi[c_{\exists x.\phi}(\vec t),\vec t]$, and therefore $\phi[S(c_{\exists x.\phi}(\vec t)),\vec t]$ is a true quantifier-free formula.
\end{proof}

Importantly, this theorem is provable in PRA: we will give a computable procedure for finding such a substitution.  First, we find simpler conditions under which $S\vDash\phi$ holds.
\begin{lemma}
  \begin{itemize}
  \item If $\phi$ is an axiom other than a critical formula then $S\vDash\phi$
  \item If $S\vDash\phi$ and $S\vDash\phi\rightarrow\psi$ then $S\vDash\psi$
  \end{itemize}
\end{lemma}

Therefore to show that $S$ satisfies the conclusion of a proof, it suffices to show that $S$ satisfies each critical formula appearing in the proof.

\begin{definition}
  Let $e$ be a closed term.
  \begin{itemize}
  \item $\phi[[v,\vec t]]:=\phi[v,\vec t]\wedge\bigwedge_{u<v}\neg\phi[u,\vec t]$
  \item $\mathcal{F}(S):=\{\phi[[v,\vec t]]\mid(c_{\exists x.\phi(x)}(\vec t),v)\in S\wedge v\neq ?\}$
  \item $S$ is \emph{correct} if for any $\phi\in\mathcal{F}(S)$, $\overline{S}\vDash\phi$
  \end{itemize}
\end{definition}

$\phi[[v]]$ just states that $v$ is the smallest value where $\phi(x)$ holds.  A correct $\epsilon$-substitution ensures that whenever it assigns a numeral to some canonical term, it is assigning the minimal correct witness to the Skolem function.

From here on, let $Cr=\{Cr_0,\ldots,Cr_N\}$ be a fixed sequence of closed critical formulas.

\begin{definition}
We say $S$ is \emph{solving} if for each $I\leq N$, $\overline{S}\vDash Cr_I$.
\end{definition}

Let $S$ be a finite, correct, nonsolving $\epsilon$-substitution.  We will consider the critical formulas made false by $\overline{S}$ and select the first one of minimal rank to be fixed.  For $I=0,\ldots,N$, we define parameters needed for the $H$-process.

\begin{definition}
  If $Cr$ is a critical formula, we define the \emph{key term}, $e(Cr)$, the \emph{parameters} $t(Cr)$, and the \emph{reduced form} $red(Cr,S)$, by:
  \begin{itemize}
  \item If $Cr$ has the form $\phi[s,\vec t]\rightarrow\phi[c_{\exists x. \phi}(\vec t),\vec t]$ then $e(Cr):=c_{\exists x. \phi}$, $t(Cr):=\vec t$, and
$$red(Cr,S):=\phi[\hat S(s),\hat S(t_1),\ldots,\hat S(t_n)]\rightarrow\phi[c_{\exists x. \phi}(\hat S(t_1),\ldots,\hat S(t_n)),\hat S(t_1),\ldots,\hat S(t_n)]$$
\item If $Cr$ has the form $\phi[0,\vec t]\wedge \neg\phi[t',\vec t]\rightarrow\phi[c_{\exists x.\phi(x)\wedge\neg\phi(\mathbf{S}x)}(\vec t),\vec t]\wedge\neg\phi[c_{\exists x.\phi(x)\wedge\neg\phi(\mathbf{S}x)}(\vec t),\vec t]$ then $e(Cr):=c_{\exists x.\phi(x)\wedge\neg\phi(\mathbf{S}x)}$, $t(Cr):=\vec t$, and
\begin{align*}
red(Cr,S):=&\phi[0,\hat S(\vec t)]\wedge \neg\phi[\hat S(t'),\hat S(\vec t))]\rightarrow\\&\phi[c_{\exists x.\phi(x)\wedge\neg\phi(\mathbf{S}x)}(\hat S(\vec t)),\hat S(\vec t)]\wedge\neg\phi[\mathbf{S}c_{\exists x.\phi(x)\wedge\neg\phi(\mathbf{S}x)}(\hat S(\vec t)),\hat S(\vec t)]
\end{align*}
\item If $Cr$ has the form $\neg s=0\rightarrow s=\mathbf{S}(c_{\exists x. y=\mathbf{S}x}(s))$ then $e(Cr):=c_{\exists x. y=\mathbf{S}x}$, $t(Cr):=s$, and
$$red(Cr,S):=\neg \hat S(s)=0\rightarrow\hat S(s)=\mathbf{S}(c_{\exists x. y=\mathbf{S}x}(\hat S(s)))$$
  \end{itemize}

  We define $e(Cr,S):=e(Cr)(\hat S(t(Cr)))$.
\end{definition}

Note that if $S\vDash\neg Cr$ where $Cr$ has the form $\phi\rightarrow\psi[c]$ then there is a fixed $u$ such that for any correct $S'\supseteq S$ deciding each $\psi[v]$ for $v\leq u$, there is some $v\leq u$ such that $S'\vDash\psi[[v]]$.

\section{Finite Injury Relationships}
We present the key idea behind the construction we will later introduce, the finite injury relationship between two trees.  We are interested in a map $\lambda$ from a tree $T_1$ to a tree $T_2$ with the property that well-foundedness of $T_2$ will guarantee well-foundedness of $T_1$.  A particularly simple way to do this would be a ``zero injury'' relationship: if $x\subsetneq y$ in $T_1$ then $\lambda(x)\subsetneq \lambda(y)$ in $T_2$.  The finite injury relationship is more flexible; in addition to allowing $\lambda(x)$ to extend $\lambda(y)$, $\lambda(x)$ might ``correct'' some choice of branch in $\lambda(y)$, but in such a way that the choice made at each node may only be ``corrected'' finitely many times.  This ensures that eventually, the choice at each node stabilizes, so an infinite branch in $T_1$ would give rise to an infinite branch in $T_2$.

For our purposes, we use a simplified form, where branches are labeled with $\mathbb{N}\cup\{?\}$ and the only possible correction is from $?$ to a value in $\mathbb{N}$.

\begin{definition}
  Let $T_1,T_2$ be trees such that the branches of $T_2$ are labeled by $\mathbb{N}\cup\{?\}$, and let $\lambda:T_1\rightarrow T_2$ be given.  We say $\lambda$ is a \emph{finite injury relationship} if whenever $x\subsetneq y$, either $\lambda(x)\subsetneq \lambda(y)$ or there is an $\alpha^\frown\langle {?}\rangle\subseteq\lambda(x)$ such that $\alpha^\frown\langle u\rangle\subseteq\lambda(y)$ for some $u\in\mathbb{N}$.
\end{definition}

Note that if we take the underlying set of $T_2$ to be partially ordered by $u<{?}$ for all $u\neq {?}$, this is the same as saying that $\lambda$ is order-preserving from the extension ordering on $T_1$ to the Kleene-Brouwer ordering on $T_2$.

\begin{definition}
  We say $\lambda:T_1\rightarrow T_2$ is \emph{weakly finite injury} if whenever $x\subseteq y$ either $\lambda(x)\subseteq \lambda(y)$ or there is an $\alpha^\frown\langle {?}\rangle\subseteq\lambda(x)$ such that $\alpha^\frown\langle u\rangle\subseteq\lambda(y)$ for some $u\in\mathbb{N}$, and for every $x$ there is a maximum $n$ such that $x_0\subsetneq x_1\subsetneq \cdots \subsetneq x_n$ implies that $\lambda(x_0)\neq \lambda(x_n)$.
\end{definition}

This weakens the finite injury condition to allow finite runs where $\lambda$ is constant.

\begin{lemma}
  If $T_2$ is well-founded and $\lambda:T_1\rightarrow T_2$ is weakly finite injury then $T_1$ is well-founded.
\end{lemma}
\begin{proof}
  Let $x_1\subsetneq x_2\subsetneq\cdots$ be an infinite branch in $T_1$.  Then we may inductively construct an infinite branch $\mu$ in $[T]$ such that for each $m$, there is some $i$ such that $\mu\upharpoonright m\subseteq\lambda(x_j)$ whenever $j\geq i$.

Suppose we have constructed $\gamma=\mu\upharpoonright m$, and let $i$ be such that $j\geq i$ implies $\gamma\subseteq\lambda(x_j)$.  Then, since $\gamma\subseteq\lambda(x_i)$ there is some $n>i$ such that $\gamma^\frown\langle u\rangle\subseteq \lambda(x_n)$ for some $u\in\mathbb{N}\cup\{?\}$.  If $u\neq{?}$ then it must be that $\gamma^\frown\langle u\rangle\subseteq\lambda(x_j)$ whenever $j>i$.

Otherwise, there are two possibilities.  Either $\gamma^\frown\langle {?}\rangle\subseteq\lambda(x_j)$ for all $j\geq n$, in which case we are done, or there is some $j>n$ such that $\gamma^\frown\langle u\rangle\subseteq\lambda(x_j)$, and therefore $\gamma^\frown\langle u\rangle\subseteq\lambda(x_k)$ whenever $k\geq j$.  In either case, we have constructed $\mu\upharpoonright m+1$.
\end{proof}

In Section \ref{Ordinals} we will show that when $\lambda$ is finite injury, an ordinal bound on the height of $T_2$ can be converted to a bound on the height of $T_1$.  (A similar argument could be made for $\lambda$ weakly finite injury, but would be made significantly more complicated by the need to handle runs where $\lambda$ is constant.)

\section{Examples of Priority Trees}
 First, we describe the general motivation behind our priority construction.  We are attempting a computation that depends on various parameters whose ``ideal'' value is non-recursive (specifically, the true values of Skolem functions).  Fortunately, we don't need to know the true value of these parameters, only values which suffice to satisfy certain conditions, the critical formulas appearing in the proof.

 Since the critical formulas contain parameters, which can themselves change in the course of our construction, a single critical formula may give rise to multiple conditions, as the values assigned to its parameters are changed.  The first step of the construction will be the process of unwinding critical formula with parameters to a tree of formulas without parameters.

 It is convenient to arrange conditions in a tree, where the nodes represent conditions and the branches representing the possible values that can be assigned to that condition.  In our case, the core conditions will turn out to be canonical Skolem terms, only some of which will be the key terms of a critical formula.  The others will be parameters needed to compute the correct way to satisfy critical formulas.

 Having built such a tree, we will proceed in stages.  At each stage we will proceed up from a node, choosing appropriate branches, until we reach a node associated with a critical formula.  At this point, we will stop and consider how to satisfy that critical formula.  We will then choose a branch, adding a Skolem term to our $\epsilon$-substitution.  This may invalidate previous choices, so we may have to backtrack; we will ensure that when we do so, we always backtrack to a node where we had chosen ${?}$, and instead choose an integer.  This ensures that our process is finite injury.

 When dealing with higher rank critical formulas---that is, questions whose ideal solution is $\Sigma_N$ for some potentially large $N$---we will have to use a tower of $N$ trees.  Roughly, the $n+1$-st tree will behave like a finite injury argument relative to the $n$-th tree: that is, as long as the $n$-th tree is simply accounting for information from the $n+1$-st, the $n+1$-st will behave in a finite injury way.  When the $n+1$-st tree reaches a level $n$-condition, the $n$-th tree may force us to throw out some information from the $n+1$-st tree and start that process over, causing the $n+1$-st tree to exhibit more complicated behavior.  So while it is difficult to describe the behavior of the $n+1$-st tree relative to the $n-1$-st directly (it is roughly that of an infinite injury argument), and essentially impossible to describe its behavior relative to the first level, we can describe each level's behavior as being finite injury relative to the previous level.

We first exhibit the simplified proof for the case where all Skolem terms have rank $1$, which substantially simplifies the process of computing a solving substitution from our construction.

\subsection{The Case of Rank $1$}
Suppose we have a set $Cr_1,\ldots,Cr_k$ of critical formulas such that $rk(e(Cr_i))=1$ for each $i\leq k$.  We first produce a tree $T_2$ branching over $\mathbb{N}\cup\{?\}$ and assign to each node $\alpha$ of height $i\leq k$ the formula $\mathit{form}(\alpha):=Cr_i$.

We fix an $\omega$-ordering $\prec$ of $T_2$ so that when $\alpha\subsetneq\beta$ then $\alpha\prec\beta$.  Next we construct another tree, $T_1$, also branching over $\mathbb{N}\cup\{?\}$.  We will assign to each node in $T_1$ other than the leaves either a critical formula whose key term is canonical or a canonical Skolem term.  Formally, for each non-leaf $\alpha\in T_1$, exactly one of $e(\alpha)$ and $\mathit{form}(\alpha)$ will be defined.  

 \begin{definition}
   Let $\alpha\in T_1$ be given.  Then $S(\alpha)$ is given inductively by:
   \begin{itemize}
   \item $S(\langle\rangle):=\emptyset$
   \item If $\mathit{form}(\alpha)$ is defined then $S(\alpha^\frown\langle u\rangle):=S(\alpha)\cup\{(e(\mathit{form}(\alpha)),u)\}$
   \item If $e(\alpha)$ is defined then $S(\alpha^\frown\langle u\rangle):=S(\alpha)\cup\{(e(\alpha),u)\}$
   \end{itemize}

   We say $\alpha$ in $T_1$ \emph{settles} a node $\beta\in T_2$ if $red(form(\beta),S(\alpha))$ contains no $\epsilon$-terms other than the key term, and $red(form(\beta),S(\alpha))$ belongs to the domain of $S(\alpha)$.
\end{definition}

Suppose we have assigned $e(\gamma)$ or $\mathit{form}(\gamma)$ to every $\gamma\subsetneq\alpha$.  Then let $\beta$ be the $\prec$-least node of $T_2$ such that $\alpha$ does not settle $\beta$.  If $red(\mathit{form}(\beta),S(\alpha))$ contains $\epsilon$-terms besides the key term, define $e(\alpha)$ to be a canonical $\epsilon$-term contained in $red(\mathit{form}(\beta),S(\alpha))$.  Otherwise, set $\mathit{form}(\alpha):=red(\mathit{form}(\beta),S(\alpha))$.

Given any node $\alpha\in T_1$, we define a path through $T_2$:
\begin{definition}
  \begin{itemize}
  \item $\langle\rangle\subseteq\lambda(\alpha)$
  \item If $\beta\subseteq\lambda(\alpha)$ and $\hat S(\alpha)(e(\mathit{form}(\beta)))$ is an integer $n$ then $\beta^\frown\langle n\rangle\subseteq\lambda(\alpha)$
  \end{itemize}
\end{definition}

Note that when $\beta\subseteq\alpha$, $\lambda(\beta)\subseteq\lambda(\alpha)$; moreover, the function cannot stabilize for infinitely long (that is, the $\lambda$ function is weakly finite injury, and furthermore, never requires the branching condition in weak finite injury).

Define a subtree $T_1'$ consisting of those nodes $\beta\in T_1$ such that for no $\gamma\subsetneq\beta$ is $\lambda(\gamma)$ a leaf in $T_2$.  That is, $\alpha$ is a leaf in $T_1'$ if it is the first node such that $\lambda(\alpha)$ is a leaf.  Then $T_1'$ is a well-founded tree.

Select a sequence of nodes through $T_1'$ as follows:
\begin{itemize}
\item $\alpha_0=\langle\rangle$
\item If $e(\alpha_n)$ is defined, set $\alpha_{n+1}:=\alpha_n^\frown\langle \overline{S(\alpha_n)}(e(\alpha_n))\rangle$
\item If $\mathit{form}(\alpha_n)$ is defined to be $\phi\rightarrow\psi[c]$ and $S(\alpha_n)\vDash\phi\rightarrow\psi[0]$ then $\alpha_{n+1}^\frown\langle \overline{S(\alpha_n)}(e(\mathit{form}(\alpha_n)))\rangle$
\item If $\mathit{form}(\alpha_n)$ is defined to be $\phi\rightarrow\psi[c]$ and $S(\alpha_n)\vDash\neg(\phi\rightarrow\psi[0])$ then there is some $n$ such that $S(\alpha_n)\vDash\psi[[n]]$.  If there is some $\gamma\subseteq\alpha_n$ such that $e(\gamma)=e(\mathit{form}(\alpha_n))$ then set $\beta:=\gamma$, otherwise set $\beta:=\alpha_n$.  Then set $\alpha_{n+1}:=\beta^\frown\langle n\rangle$
\end{itemize}

This is a finite injury process from the integers to $T_1'$, and therefore terminates at some node $\alpha$.  Observe that $S(\alpha)$ is correct and satisfies every critical formula along the path up to $\lambda(\alpha)$, and is therefore a solving substitution.

\section{The Main Construction}

\subsection{Trees}
Let $Cr_1,\ldots,Cr_K$ be a fixed sequence of critical formulas.  In the tree $T_{N+1}$, assign to each node $\alpha$ the critical formula $Cr_I$ where $I$ is the height of $\alpha$; denote this by $\mathit{form}(\alpha)$.

To each node other than leaves in the trees $T_1,\ldots,T_N$, we will assign either a canonical Skolem term of rank $\leq N$, which we will denote $e(\alpha)$, or a critical formula with canonical key term, which we will denote $\mathit{form}(\alpha)$.

To every node $\alpha$ in a tree $T_i$, $i\leq N$, we assign an $\epsilon$-substitution $S(\alpha)$:
\begin{definition}
  \begin{itemize}
  \item $S(\langle\rangle):=\emptyset$
  \item If $\mathit{form}(\alpha)$ is defined then $S(\alpha^\frown\langle u\rangle):=S(\alpha)\cup\{(e(\mathit{form}(\alpha)),u)\}$
  \item If $e(\alpha)$ is defined then $S(\alpha^\frown\langle u\rangle):=S(\alpha)\cup\{(e(\alpha),u)\}$
  \end{itemize}
\end{definition}

Suppose $T_{i+1}$ has been given.  Then fix a constructive $\omega$-ordering $\prec$ of $T_{i+1}$ with the property that if $\alpha\subsetneq\beta$ then $\alpha\prec\beta$.

\begin{definition}
We say a node $\alpha$ in $T_i$ \emph{settles} a node $\beta$ in $T_{i+1}$ if one of the following holds:
\begin{itemize}
\item $e(\beta)$ is defined, $rk(e(\beta))\leq i$, and there is some $\gamma^\frown\langle u\rangle\subseteq\alpha$ with $e(\gamma)=e(\beta)$
\item $e(\beta)$ is defined and $rk(e(\beta))>i$
\item $\mathit{form}(\beta)$ is defined, $rk(\mathit{form}(\beta))\leq i$, and there is some $\gamma^\frown\langle u\rangle\subseteq\alpha$ with $e(\mathit{form}(\gamma))=e(\mathit{form}(\beta))$ or $e(\gamma)=e(\mathit{form}(\beta))$
\item $\mathit{form}(\beta)$ is defined, $rk(\mathit{form}(\beta))>i$, $\mathit{form}(\beta)$ has the form $\phi\rightarrow\psi[c]$, and $S(\alpha)\vDash\phi\rightarrow\psi[0]$
\item $\mathit{form}(\beta)$ is defined, $rk(\mathit{form}(\beta))>i$, $\mathit{form}(\beta)$ has the form $\phi\rightarrow\psi[c]$, $S(\alpha)\vDash\neg(\phi\rightarrow\psi[0])$, and $S(\alpha)\vDash\psi[[n]]$ for some $n$
\end{itemize}
\end{definition}

Suppose we have assigned $\mathit{form}(\alpha)$ or $e(\alpha)$ for every $\alpha\subsetneq\beta$ in $T_i$.  Let $\beta$ be the $\prec$-least element of $T_{i+1}$ such that $\alpha$ does not settle $\beta$.  If $e(\beta)$ is defined then $rk(e(\beta))\leq i$, and we set $e(\alpha):=e(\beta)$.  Otherwise $\mathit{form}(\beta)$ is defined; if $\mathit{form}(\beta)\leq i$ and $red(\mathit{form}(\beta),S(\alpha))$ contains no $\epsilon$-terms besides the key term then set $\mathit{form}(\alpha):=red(\mathit{form}(\beta),S(\alpha))$.  Otherwise set $e(\alpha)$ to be a canonical $\epsilon$-term other than the key term appearing in $red(\mathit{form}(\beta),S(\alpha))$.  If $\mathit{form}(\beta)>i$ and $S(\alpha)$ does not decide $\phi\rightarrow\psi[0]$ then let $e(\alpha)$ be an element of $unev(\phi\rightarrow\psi[0],S(\alpha))$.  If $S(\alpha)\vDash\neg(\phi\rightarrow\psi[0])$ but there is no $n$ such that $S(\alpha)\vDash\psi[[n]]$ then let $n$ be least such that $S(\alpha)$ does not decide $\phi[n]$ and let $e(\alpha)$ be an element of $unev(\phi[n],S(\alpha))$.

Note that if $\alpha$ does not settle $\beta$, $\alpha^\frown\langle u\rangle$ may not settle $\beta$ either.  However it takes only finitely many extensions to settle $\beta$, and therefore along any path through $T_i$, every $\beta\in T_{i+1}$ is settled.

Also, note that the requirements about $red(\mathit{form}(\beta),S(\alpha))$ are necessary only when forming $T_{N}$: on lower trees $\mathit{form}(\beta)$ already contains no $\epsilon$-terms besides the key term.

\subsection{Building a Solving Substitution}
Now we describe the actual construction of a particular solving substitution, using the trees $T_1,\ldots,T_{N+1}$.

\begin{definition}
  For a node $\alpha\in T_i$, we define a sequence of nodes $\alpha^n_+$ in $T_{i+1}$ by recursion as follows.  $\alpha^0_+$ is $\langle\rangle$.  If $\alpha$ does not settle $\alpha^n_+$ or $\alpha^n_+$ is a leaf then the process terminates.  Otherwise we split into cases.

If $e:=e(\alpha^n_+)$ is defined and $rk(e)\leq i$ then $S(\alpha)(e)$ is defined and $\alpha^{n+1}_+:=\alpha^n_+{}^\frown\langle S(\alpha)(e)\rangle$.  If $rk(e)>i$ and there is some $u$ such that $S(\alpha)\vDash\phi[[u,\vec t]]$ where $e(\alpha)$ is $c_{\exists x. \phi[x,\vec y]}(\vec t)$ then $\alpha^{n+1}_+:=\alpha^n_+{}^\frown\langle u\rangle$, and if there is no such $u$, $\alpha^{n+1}_+:=\alpha^n_+{}^\frown\langle {?}\rangle$.

If $f:=\mathit{form}(\alpha^n_+)$ is defined and $rk(f)\leq i$ then $S(\alpha)(e(f))$ is defined and $\alpha^{n+1}_+:=\alpha^n_+{}^\frown\langle S(\alpha)(e(f))\rangle$.  Otherwise $rk(f)>i$, and either $S(\alpha)\vDash \phi\rightarrow\psi[c]$, in which case $\alpha^{n+1}_+:=\alpha^n_+{}^\frown\langle {?}\rangle$, or there is a $u$ such that $S(\alpha)\vDash\psi[[u]]$.  In this case, if there is some $\gamma\subseteq\alpha^n_+$ such that $e(\gamma)=e(\mathit{form}(\alpha^n_+),S(\alpha))$ then set $\beta:=\gamma$, otherwise set $\beta:=\alpha^n_+$.  Set $\alpha^{n+1}_+:=\beta^\frown\langle u\rangle$.
\end{definition}

Define $\delta_N:T_N\rightarrow T_{N+1}$ to be the final $\alpha^n_+$ in this process; note that this exists since there are no infinite paths through $T_{N+1}$ and the process is weakly finite injury, and therefore terminates.  Define $T_N'$ to be those nodes $\alpha\in T_N$ such that for no $\beta\subsetneq\alpha$ is $\delta_N(\beta)$ a leaf---that is, leaves in $T_N'$ are initial nodes $\alpha$ such that $\delta_N(\alpha)$ is a leaf.  This tree is well-founded since $T_{N+1}$ is and $\delta_N$ restricted to $T'_N$ is weakly finite injury.

We may iterate this process, defining subtress $T_i'\subseteq T_i$ and weakly finite injury maps $\delta_i:T_i'\rightarrow T_{i+1}'$.

Now choose a path through $T_1'$ as follows:
\begin{itemize}
\item Define $\alpha_0$ to be $\langle\rangle$
\item If $e(\alpha_n)$ is defined, set $\alpha_{n+1}:=\alpha_n^\frown\langle \overline{S(\alpha_n)}(e(\alpha_n))\rangle$
\item If $\mathit{form}(\alpha_n)$ is defined to be $\phi\rightarrow\psi[c]$ and $S(\alpha_n)\vDash\phi\rightarrow\psi[0]$ then $\alpha_{n+1}^\frown\langle \overline{S(\alpha_n)}(e(\mathit{form}(\alpha_n)))\rangle$
\item If $\mathit{form}(\alpha_n)$ is defined to be $\phi\rightarrow\psi[c]$ and $S(\alpha_n)\vDash\neg(\phi\rightarrow\psi[0])$ then there is some $n$ such that $S(\alpha_n)\vDash\psi[[n]]$.  If there is some $\gamma\subseteq\alpha_n$ such that $e(\gamma)=e(\mathit{form}(\alpha_n))$ then set $\beta:=\gamma$, otherwise set $\beta:=\alpha_n$.  Then set $\alpha_{n+1}:=\beta^\frown\langle n\rangle$
\end{itemize}

This process is a finite injury function from the natural numbers to $T_1'$, and therefore terminates after finitely many steps at some node $\alpha$.

\begin{lemma}
  $S(\alpha_n)$ is correct for each $n$.
\end{lemma}
\begin{proof}
    By induction on $n$.  $S(\alpha_0)$ is empty, and therefore correct.  If $(\epsilon x.\phi[x],u)\in S(\alpha_{n+1})$ with $u\neq{?}$ then either $(\epsilon x.\phi[x],u)\in S(\alpha_n)$ or $S(\alpha_n)\vDash\phi[[u]]$; in either case, since $(S(\alpha_n))_{<rk(\epsilon x.\phi[x])}\subseteq S(\alpha_{n+1})$ and $S(\alpha_n)$ is correct, $S(\alpha_{n+1})\vDash\phi[[u]]$.
\end{proof}

\begin{lemma}
  If $S(\gamma)$ is correct then $S(\gamma^n_+)$ is correct for all $n$.  In particular, $\delta_i(\gamma)$ is correct.
\end{lemma}
\begin{proof}
  By induction on $n$.  Essentially the same as the previous lemma.
\end{proof}

Then in particular, let $\gamma:=\delta_{N-1}(\cdots \delta_1(\alpha)\cdots)$.  $S(\gamma)$ is correct, and therefore a solving substitution.

\section{Ordinal Analysis}
\label{Ordinals}

\begin{lemma}
  Suppose $f:T_1\rightarrow T_2$ is a finite injury relation, and $T_2$ has height $\alpha$.  Then there is a height function $o:T_1\rightarrow\omega^\alpha$ such that $x\subseteq y$ implies $o(y)<o(x)$.
\end{lemma}
\begin{proof}
  Let $h:T_2\rightarrow\alpha$ be such that $s\subsetneq t$ implies $h(t)< h(s)$.  Then we define $o:T_1\rightarrow\omega^\alpha$ as follows:
$$o(x)=\left(\sum_{a\frown\langle {?}\rangle\subseteq f(x)}\omega^{h(a)}\right)+\omega^{h(x)+1}$$

We must show that this is order-preserving.  Let $y\subsetneq x$, and suppose $f(y)\subsetneq f(x)$.  Then we have $f(y)\frown\langle u\rangle\subseteq f(x)$; if $u={?}$ then
$$o(x)=\left(\sum_{a\frown\langle {?}\rangle\subseteq f(y)}\omega^{h(a)}\right)+\omega^{h(f(y))}+\left(\sum_{f(y)\frown\langle {?}\rangle\subseteq a\frown\langle {?}\rangle\subseteq f(x)}\omega^{h(a)}\right)+\omega^{h(f(x))+1}$$

So it suffices to show that
$$\omega^{h(f(y))}+\left(\sum_{f(y)\frown\langle {?}\rangle\subsetneq a\frown\langle {?}\rangle\subseteq f(x)}\omega^{h(a)}\right)+\omega^{h(f(x))+1}<\omega^{h(y)+1}$$
But this is clear, since $h(f(x))+1\leq h(f(y))<h(f(y))+1$ and $h(a)<h(f(y))$ whenever $f(y)\frown\langle {?}\rangle\subseteq a$.

If $u\neq {?}$ then this is even simpler, since the $\omega^{h(f(y))}$ term is omitted.

Now suppose that $b\frown\langle {?}\rangle\subseteq f(y)$ and $a\frown\langle u\rangle\subseteq f(x)$.  Then
$$o(y)=\beta+\omega^{h(f(y))}+\gamma$$
for suitable $\gamma<\omega^{h(f(y))}<\beta$, and
$$o(x)=\beta+\delta$$
where $\delta<\omega^{h(f(y))}$.  Therefore $o(x)<o(y)$.
\end{proof}

\bibliography{PriPABib}
\end{document}